\documentclass[sn-mathphys]{sn-jnl}

\jyear{2021}%

\theoremstyle{thmstyleone}%
\newtheorem{theorem}{Theorem}
\newtheorem{corollary}[theorem]{Corollary}
\newtheorem{lemma}[theorem]{Lemma}

\theoremstyle{thmstyletwo}%
\newtheorem{example}{Example}%
\newtheorem{remark}{Remark}%

\theoremstyle{thmstylethree}%
\newtheorem{definition}{Definition}%

\newcommand{\cdummy}{\cdot}
\newcommand{\assign}{:=}
\newcommand{\tmop}[1]{\ensuremath{\operatorname{#1}}}
\newcommand{\tmtextrm}[1]{{\rmfamily{#1}}}

\begin{document}

\title[Low-rank approximation of CPS tensor]{The low-rank approximation of fourth-order partial-symmetric and conjugate partial-symmetric tensor}


\author[1,2]{\fnm{Amina} \sur{Sabir}}\email{amina\_sabir1@126.com}

\author*[2]{\fnm{Peng-fei} \sur{Huang}}\email{huangpf@mail.nankai.edu.cn}

\author[2]{\fnm{Qing-zhi} \sur{Yang}}\email{qz-yang@nankai.edu.cn}

\affil[1]{\orgdiv{School of Mathematics and Statistics}, \orgname{Kashi University}, \orgaddress{\city{Kashi}, \postcode{844008}, \country{P.R. China}}}

\affil[2]{\orgdiv{School of Mathematical
Sciences and LPMC}, \orgname{Nankai University}, \orgaddress{\city{Tianjin}, \postcode{300071}, \country{P.R. China}}}


\abstract{We present an orthogonal matrix outer product decomposition for the fourth-order conjugate partial-symmetric (CPS) tensor and show that the greedy successive rank-one approximation (SROA) algorithm can recover this decomposition exactly. Based on this matrix decomposition, the CP rank of CPS tensor can be bounded by the matrix rank, which can be applied to low rank tensor completion. Additionally, we give the rank-one equivalence property for the CPS tensor based on the SVD of matrix, which can be applied on the rank-one approximation for CPS tensors.}




\maketitle

\section{Introduction}\label{sec1}

Tensor decomposition and approximation have significant applications in computer vision, data mining, statistical estimation and so on. We refer to \cite{kolda2009tensor} for the survey. Moreover, it is general that tensor from real applications with special symmetric structure. For instance, the symmetric outer product decomposition is particularly important in the process of blind identification of under-determined mixtures \cite{comon2008symmetric}.

Jiang et al.\cite{jiang2016characterizing} studied the functions in multivariate complex variables which always take real values. They proposed the conjugate partially symmetric (CPS) tensor to characterize such polynomial functions, which is the generalization of Hermitian matrix. Various examples of conjugate partial-symmetric tensors can be encountered in engineering applications arising from singal processing, electrical engineering, and control theory\cite{aubry2013ambiguity,de2007fourth}. Ni et al.\cite{Ni2019Hermitian} and Nie et al.\cite{Nie2019Hermitian} researched on the Hermitian tensor decomposition. Motivated by Lieven et al.\cite{de2007fourth}, we proposed a new orthogonal matrix outer product decomposition model for CPS tensors, which explore the orthogonality of these matrices.

It is well known that unlike the matrix case, the best rank-r ($r>1$) approximation of a general tensor may not exist, and even if it admits a solution, it is NP-hard to solve\cite{de2008tensor}. The greedy successive rank-one approximation (SROA) algorithm can be applied to compute the rank-r ($r>1$) approximation of tensor. However, the theoretical guarantee for obtaining the best rank-r approximation is less developed. Zhang et al.\cite{zhang2001rank} first proved the successive algorithm exactly recovers the symmetric and orthogonal decomposition of the underlying real symmetrically and orthogonally decomposable tensors. Fu et al.\cite{2018Successive} showed that SROA algorithm can exactly recover unitarily decomposable CPS tensors. We offer the theoretical guarantee of SROA algorithm for our matrix decomposition model of CPS tensor tensors.

Many multi-dimensional data from real practice are fourth-order tensors and can be formulated as low-CP-rank tensor. However it is very difficult to compute CP-rank of a tensor. Jiang et al.\cite{Jiang2018Low} showed that CP-rank can be bounded by the corresponding rank of square unfolding matrix of tensors. Following the idea we research on the low rank tensor completion for fourth order partial-symmetric tensor in particular.

Recently, Jiang et al. \cite{2015Tensor} proposed convex relaxations for solving a tensor optimization problem closely related to the best rank-one approximation problem for symmetric tensors. They proved an equivalence property between a rank-one symmetric tensor and its unfolding matrix. Yang et al. \cite{Yuning2016Rank} studied the rank-one equivalence property for general real tensors. Based on these rank-one equivalence properties, the above mentioned tensor optimization problem can be casted into a matrix optimization problem, which alleviates the difficulty of solving the tensor problem. In line with this idea, we study the rank-one equivalence property for the fourth-order CPS tensor and transform the best rank-one tensor approximation problem into a matrix optimization problem.

In Section \ref{sec:pre}, we give some notations and definitions. The outer product approximation model based on matrix is proposed and the successive rank-one approximation (SMROA) algorithm is given to solve it in Section \ref{sec:model}. We show that the SMROA algorithm can exactly recover the matrix outer product decomposition or approximation of the CPS tensor in Section \ref{sec:theory}. Section \ref{sec:application} discusses applications of our model simply. In Section \ref{sec:rank1Equ}, we present the rank-one equivalence property of fourth-order CPS tensor, and based on it an application is proposed. Numerical examples are in Section \ref{sec:numerical}.

\section{Preliminary}\label{sec:pre}
All tensors in this paper are fourth-order. For any complex number $z=a+ib\in C$, $\bar{z}=a-ib$ denotes the conjugate of $z$. "$\circ$" denotes the outer product of matrices, namely $\mathcal{A}=X\circ Y$ means that
\[\mathcal{A}_{ijkl}=X_{ij}Y_{kl}.\]
$S^n$ denotes the set of $n$ by $n$ symmetric matrices, the entries of these matrices can be complex or only real according to the context, without causing ambiguity. The inner product between $\mathcal{A},~\mathcal{B}\in C^{n^4}$ is defined as
\[\left<\mathcal{A},~\mathcal{B}\right>=\sum\limits_{i,j,k,l=1}^{n}\mathcal{A}_{ijkl}\bar{\mathcal{B}}_{ijkl}.\]
\begin{definition}
A fourth-order tensor $\mathcal{A}$ is called symmetric if $\mathcal{A}$ is invariant under all permutations of its indices, i.e.,
\begin{equation*}
\mathcal{A}_{ijkl}=\mathcal{A}_{\pi(ijkl)},\quad i,~j,~k,~l=1,\cdots n.
\end{equation*}
\end{definition}
\begin{definition}\cite{Ni2019Hermitian}
A fourth-order complex tensor $\mathcal{A}\in C^{n_1\times n_2\times n_1\times n_2}$ is called a Hermitian tensor if
\[\mathcal{A}_{i_1i_2j_1j_2}=\bar{\mathcal{A}}_{j_1j_2i_1i_2}.\]
\end{definition}
Jiang et al.\cite{jiang2016characterizing} introduced the concept of conjugate partial-symmetric tensors as following.
\begin{definition}
A fourth-order complex tensor $\mathcal{A}\in C^{n^4}$ is called conjugate partial-symmetric (CPS) if
\begin{equation*}
\begin{aligned}
\mathcal{A}_{ijkl}&=\mathcal{A}_{\pi(ij)\pi(kl)}\\
\mathcal{A}_{ijkl}&=\overline{\mathcal{A}}_{klij},\quad i,~j,~k,~l=1,\cdots n.
\end{aligned}
\end{equation*}
\end{definition}
\begin{definition}
A fourth-order tensor $\mathcal{A}\in R^{n^4}$ is called partial-symmetric if
\begin{equation*}
\mathcal{A}_{ijkl}=\mathcal{A}_{\pi(ij)\pi(kl)}=\mathcal{A}_{\pi(kl)\pi(ij)},\quad i,~j,~k,~l=1,\cdots n.
\end{equation*}
\end{definition}
\begin{example}\cite{de2007fourth}
In the blind source separation problem, the cumulant tensor is computed as
\[\mathcal{C}=\sum\limits_{r=1}^Rk_ra_r\circ \bar{a}_r\circ \bar{a}_r\circ a_r.\]
By a permutation of the indices, it is in fact a conjugate partial-symmetric tensor.
\end{example}
\begin{definition}
The square unfolding form $M(\mathcal{A})$ of a fourth-order tensor $\mathcal{A}\in C^{n^4}$ is defined as
\begin{equation*}
M(\mathcal{A})_{(j-1)n+i,(l-1)n+k}=\mathcal{A}_{ijkl}.
\end{equation*}
\end{definition}

\section{Matrix outer product approximation model}\label{sec:model}
Jiang et al.\cite{Jiang2018Low} introduced the new notions of M-decomposition for an even-order tensor $\mathcal{A}$, which is exactly the rank-one decomposition of $M(\mathcal{A})$, followed by the notion of tensor M-rank.

For each $\mathcal{A}\in R^{n^4}$, let $M(\mathcal{A})=U\Sigma V^T=\sum\limits_{i=1}^{n^2}\sigma_iu_i\circ v_i$ be the SVD of $M(\mathcal{A})$, then $\mathcal{A}$ has the following decomposition form
\begin{equation}\label{equ:general}
\mathcal{A}=\sum\limits_{i=1}^{n^2}\sigma_iU_i\circ V_i,
\end{equation}
where $(U_i)_{st}=(u_i)_{(t-1)n+s}$, $(V_i)_{st}=(v_i)_{(t-1)n+s}$, for $i=1,2,\cdots,n^2$. $\left<U_i,U_j\right>=\delta_{ij},~\left<V_i,V_j\right>=\delta_{ij}$, $i,j=1,2,\cdots,n^2$.

We are particularly interested in the tensor with some symmetric properties. And analogous to Lieven et al.\cite{de2007fourth}, we prove that the CPS tensor has a decomposition based on matrix as following.
\begin{theorem}\label{thm:main}
If $\mathcal{A}\in C^{n^4}$ is a conjugate partial-symmetric tensor, then it can be decomposed as,
\begin{equation}
\mathcal{A}=\sum\limits_{i=1}^r\lambda_iE_i\circ \overline{E_i},
\end{equation}
where $\lambda_i\in R$, $E_i\in C^{n^2\times n^2}$ are symmetric matrices and $\left<E_i,E_j\right>=\delta_{ij}$, for $i,~j=1,2,\cdots,r$. And the decomposition is unique when $\lambda_i$ are different from each other.
\end{theorem}
\begin{proof}
Since $\mathcal{A}$ is conjugate partial-symmetric, then the unfold matrix $M(\mathcal{A})$ is Hermitian, and can be decomposed as
\[
M(\mathcal{A}) = \sum_{i = 1}^r \lambda_i e_i e_i^*,\]
where $\lambda_i\in R$, $e_i\in C^{n^2}$, $i = 1,\ldots, r$  are mutually orthogonal. Folding $e_i$ into matrix $E_i$ via $(E_i)_{ij} = (e_i)_{(j-1)\times n + i}$, thus $E_i$, $i = 1,\ldots, p$ are  mutually orthogonal, that is $\langle E_i, E_j \rangle = \delta_{ij}$. In this case, we have $\mathcal{A} = \sum\limits_{i = 1}^r \lambda_i E_i\circ \bar{E_i}$.

From the eigen-decomposition of $M(\mathcal{A})$, we have $M(\mathcal{A})e_\tau = \lambda_\tau e_\tau$, for $\tau = 1,\ldots, r$ i.e., $\sum_{k,l}a_{ijkl}(e_\tau)_{(l-1)\times n + k} = \lambda_\tau (e_\tau)_{(j-1)\times n + i}$, for any $i, j = 1,\ldots, n$. Since $a_{ijkl} = a_{jikl}$, for all $i, j = 1,\ldots, n$, then $(e_\tau)_{(j-1)\times n + i} =(e_\tau)_{(i-1)\times n + j} $, thus $E_\tau$ is symmetric. The uniqueness of the decomposition follows the property of eigen-decomposition of Hermitian matrix naturally.
\end{proof}
\begin{remark}
Jiang et al.\cite{jiang2016characterizing} gave the decomposition theorem for CPS tensor like Theorem \ref{thm:matrixDecomp}. However, they established this theorem in the view of polynomial decomposition and did not explore the mutually orthogonality of matrices in the decomposition model.
\end{remark}
\begin{definition}
$\mathcal{A}\in C^{n^4}$ is a CPS tensor,
$$rank_M(\mathcal{A})=min\{r\mid\mathcal{A}=\sum\limits_{i=1}^r\lambda_iE_i\circ \bar{E}_i,\lambda_i\in R,E_i\in S^{n^2}\}.$$
\end{definition}
The $rank_M(\mathcal{A})$ is actually the strongly symmetric M-rank $rank_{ssm}(\mathcal{A})$ defined by Jiang\cite{Jiang2018Low}. For symmetric tensor $\mathcal{A}$, they also proved the equivalence between $rank_{ssm}(\mathcal{A})$ and
$$rank_{sm}(\mathcal{A})=min\{r\mid\mathcal{A}=\sum\limits_{i=1}^r\lambda_iE_i\circ E_i,\lambda_i\in R,E_i\in C^{n^2}\}.$$ This is also true for CPS fourth-order tensor.
\begin{theorem}
Let $\mathcal{A}\in C^{n^4}$ be a CPS tensor, then $rank_M(\mathcal{A})=rank_{sm}(\mathcal{A}).$
\end{theorem}
\begin{proof}
It is obvious that $rank_M(\mathcal{A})\ge rank_{SM}(\mathcal{A}).$ On the other hand, if $rank_{SM}(\mathcal{A})=r$, we have $rank(M(\mathcal{A}))\le r$. Since $rank(M(\mathcal{A}))=rank_M(\mathcal{A})$, we obtain the desired conclusion.
\end{proof}
\begin{corollary}\label{thm:matrixDecomp}
Let $\mathcal{A}\in R^{n^4}$ be a partial-symmetric tensor, then one has,
\begin{equation}
\mathcal{A}=\sum\limits_{i=1}^{r}\lambda_iE_i\circ E_i,
\end{equation}
where $\lambda_i\in R$, $E_i$ are symmetric matrices and $\left<E_i,E_j\right>=\delta_{ij}$, for $i,~j=1,2,\cdots,r$. $rank_M(\mathcal{A})\le\frac{n(n+1)}{2}$.
\end{corollary}
\begin{proof}
The first part is obvious according to Theorem \ref{thm:matrixDecomp}. Since all matrices belonging to $S^n$ form a $\frac{n(n+1)}{2}$-dimensional vector space, we have $rank_M(\mathcal{A})\le\frac{n(n+1)}{2}$.
\end{proof}
Fu et al. gave a rank-one decomposition of vector form for the CPS tensor based on Theorem \ref{thm:main} as follows,
\begin{theorem}\cite[Theorem 3.2]{fu2018decompositions}
$\mathcal{A}\in C^{n^4}$ is CPS if and only if it has the following partial-symmetric decomposition
\[\mathcal{A}=\sum\limits_i\lambda_i\bar{a}_i\circ \bar{a}_i\circ a_i\circ a_i,\]
where $\lambda_j\in R$ and $a_i\in C^n$. That is, a CPS tensor can be decomposed as the sum of rank-one CPS tensors.
\end{theorem}
However, when we restricted the decomposition on real domain, the decomposition does not seem to hold, since $\sum\limits_i\lambda_i a_i\circ a_i\circ a_i\circ a_i$, where $\lambda_i\in R$, $a_i\in R^n$, can only represent symmetric tensor. Thus, an extended rank-one approximation model for the partial-symmetric tensor can be proposed based on Corollary \ref{thm:matrixDecomp}.

\begin{corollary}\label{thm:vecdecomp}
Let $\mathcal{A}\in R^{n^4}$ be a partial-symmetric tensor, then it can be decomposed as the sum of simple low rank partial-symmetric tensor,
\begin{equation}\label{equ:vecdecomp}
\mathcal{A}=\sum\limits_i\lambda_i(p_i\circ p_i\circ q_i\circ q_i+q_i\circ q_i\circ p_i\circ p_i).
\end{equation}
\end{corollary}
\begin{proof}
From Corollary \ref{thm:matrixDecomp}, partial-symmetric tensor $\mathcal{A} = \sum\limits_{i = 1}^r \lambda_i E_i\circ E_i$, where $E_i$ are symmetric. So it can be decomposed as $\sum_{j = 1}^{r_i}\beta_i^j u_i^j(u_i^j)^\top$, thus
\[
\begin{aligned}
\mathcal{A} &= &&\sum\limits_{i = 1}^r \lambda_i (\sum_{j = 1}^{r_i} \beta_i^j u_i^j(u_i^j)^\top)\circ (\sum_{k = 1}^{r_i} \beta_i^k u_i^k(u_i^k)^\top)  \\
                  & = &&\sum_{i = 1}^r \lambda_i (\sum_{j = 1}^{r_i}\sum_{k = j}^{r_i} \beta_i^j \beta_i^k(u_i^j \circ u_i^j \circ u_i^k \circ u_i^k + u_i^k \circ u_i^k \circ u_i^j \circ u_i^j).
\end{aligned}
\]
The desired decomposition form follows.
\end{proof}
\begin{remark}
From the proof of Corollary \ref{thm:vecdecomp}, we can see that if $p_i\neq q_i$, $p_i^Tq_i=0$. Whether this decomposition form is the compactest will be one of our future work.
\end{remark}

We can discuss the case of skew partial-symmetric tensor in parallel.
\begin{theorem}\label{thm:skew}
We call $\mathcal{A}\in R^{n^4}$ skew partial-symmetric tensor if
\begin{equation*}
\mathcal{A}_{ijkl}=\mathcal{A}_{\pi(ij)\pi(kl)}=-\mathcal{A}_{\pi(kl)\pi(ij)},\quad i,~j,~k,~l=1,2,\cdots, n.
\end{equation*}
Then one has
\begin{equation*}
\mathcal{A}=\sum\limits_{i}\lambda_i(U_i\circ V_i-V_i\circ U_i),
\end{equation*}
and
\begin{equation*}
\mathcal{A}=\sum\limits_i\lambda_i(p_i\circ p_i\circ q_i\circ q_i-q_i\circ q_i\circ p_i\circ p_i).
\end{equation*}
\end{theorem}
\begin{proof}
$M(\mathcal{A})$ is skew-symmetric according to the definition of the skew partial-symmetric tensor. Then $M(\mathcal{A})=\sum_{i}\lambda_i(u_iv_i^T-v_iu_i^T)$. The rest of the proof is similar to that for partial-symmetric tensor, here we omit it.
\end{proof}

Based on Theorem \ref{thm:main}, we propose a matrix outer product approximation model for the CPS tensor as following.
\begin{equation}\label{equ:model_matrix}
\begin{aligned}
\underset{\lambda_i\in R,~X_i\in S^n}{min}&\|\mathcal{A}-\sum\limits_{i=1}^r \lambda_iX_i\circ \bar{X_i}\|^2_F\\
s.t. & \left<X_i,X_j\right>=\delta_{ij}.
\end{aligned}
\end{equation}

The successive rank-one approximation algorithm can also be applied to the conjugate partial symmetric tensors to find the matrix outer product decompositions or approximations, as shown in Algorithm \ref{alg:SMROA}.

\begin{algorithm}
    \caption{Successive Matrix Outer Product Rank-One Approximation (SMROA) Algorithm}\label{alg:SMROA}
     Given a CPS tensor $\mathcal{A}\in C^{n^4}$. Initialize $\mathcal{A}_0=\mathcal{A}$.
    \begin{algorithmic}
    \For {$j=1$ to $r$}
        \State $(\lambda_j,X_j)\in \underset{\|X\|_F=1,X\in S^n,\lambda\in R}{argmin}\|\mathcal{A}_{j-1}-\lambda X\circ \bar{X}\|_F$.
        \State $\mathcal{A}_{j}=\mathcal{A}_{j-1}-\lambda_j X_j\circ\bar{X_j}$.
    \EndFor
    \State Return $\{\lambda_j,X_j\}_{j=1}^{r}$
    \end{algorithmic}
\end{algorithm}

The main optimization problem in Algorithm \ref{alg:SMROA} could be expressed as
\begin{equation}\label{equ:onestep}
(\lambda_*,X_*)\in \underset{\|X\|_F=1,X\in S^n,\lambda\in R}{argmin}\|\mathcal{A}-\lambda X\circ \bar{X}\|_F^2,
\end{equation}
The objective function of \eqref{equ:onestep} can be rewritten as
\begin{equation*}
\begin{split}
&\|\mathcal{A}-\lambda X\circ \bar{X}\|_F^2\\
=&\|\mathcal{A}\|_F^2+\lambda^2-2\lambda\left<\mathcal{A},X\circ \bar{X}\right>
\end{split}
\end{equation*}
From which we can derive that problem \eqref{equ:onestep} is equivalent to
\begin{equation}\label{equ:main1}
X_*\in\underset{\|X\|_F=1,X\in S^n}{argmax}\mid\left<\mathcal{A},X\circ\bar{X}\right>\mid,
\end{equation}
and $\lambda_*=\left<\mathcal{A},X_*\circ\bar{X_*}\right>$.
We can solve \eqref{equ:main1} by transforming it into matrix eigenproblem as follows,
\begin{equation}\label{matrix_eigen}
x_*\in \underset{\|x\|=1,x\in C^{n^2}}{argmax}\mid x^*M(\mathcal{A})x\mid.
\end{equation}
\begin{remark}
Zhang et al.\cite{zhang} proved that if $\mathcal{A}\in R^{n^4}$ is symmetric,
\[\underset{\substack{\|x_i\|=1\\i=1,2,3,4}}{min}\|\mathcal{A}-\lambda x_1\circ x_2\circ x_3\circ x_4\|_F=\underset{\|x\|=1}{min}\|\mathcal{A}-\lambda x\circ x\circ x\circ x\|_F;\]
if $\mathcal{A}$ is symmetric about the first two and the last two mode respectively,
\[\underset{\substack{\|x_i\|=1\\i=1,2,3,4}}{min}\|\mathcal{A}-\lambda x_1\circ x_2\circ x_3\circ x_4\|_F=\underset{\|x\|=\|y\|=1}{min}\|\mathcal{A}-\lambda x\circ x\circ y\circ y\|_F.\]
It is obvious that for partial-symmetric tensor, we also have
\[\underset{\|X_i\|_F=1,X_i\in R^{n^2},\lambda\in R}{min}\|\mathcal{A}-\lambda X_1\circ X_2\|_F=\underset{\|X\|_F=1,X\in S^n,\lambda\in R}{min}\|\mathcal{A}-\lambda X\circ X\|_F.\]
\end{remark}
\begin{remark}
It is well-known that \eqref{equ:onestep} is equivalent to the nearset Kronecker product problem \cite{golub2013matrix} as below
\begin{equation*}
(\lambda_*,X_*)\in \underset{\|X\|_F=1,X\in S^n,\lambda\in R}{argmin}\|A-\lambda X\otimes \bar{X}\|_F^2,
\end{equation*}
where $A_{(i-1)n+k,(j-1)n+l}=\mathcal{A}_{ijkl}$, "$\otimes$" denotes the kronecker product of matrices.
\end{remark}
\section{Exact Recovery for CPS tensors}\label{sec:theory}
In this section, we give the theoretical analysis of exact recovery for CPS tensors by the SMROA algorithm.
\begin{theorem}\label{thm:recover}
Let $\mathcal{A}$ be a CPS tensor with $rank_{M}(\mathcal{A})=r$, that is  \[\mathcal{A}=\sum\limits_{i=1}^r\lambda_iE_i\circ\bar{E_i}.\]
If $\lambda_i$ are different from each other, then the SMROA algorithm will obtain the exact decomposition of $\mathcal{A}$ after $r$ iterations.
\end{theorem}
We first claim the following lemma before proving the above theorem.
\begin{lemma}\label{lem:onestep}
Let $\mathcal{A}$ be a CPS tensor with $rank_{M}(\mathcal{A})=r$, that is
\[\mathcal{A}=\sum\limits_{i=1}^r\lambda_iE_i\circ\bar{E_i}.\]
$\lambda_i$ are different from each other. Suppose
\begin{equation*}
\hat{X}_1\in\underset{X\in S^n, \|X\|_F=1}{argmax}\mid\left<\mathcal{A},X\circ\bar{X}\right>\mid,
~\hat{\lambda}_1=\left<\mathcal{A},X\circ\bar{X}\right>.
\end{equation*}
Then, there exists $j\in\{1,2,\cdots,r\}$ such that
\[\hat{\lambda}_1=\lambda_j,~\hat{X}_1=E_j.\]
\end{lemma}
\begin{proof}
According to Theorem \ref{thm:main}, $E_i$ are mutually orthogonal, thus $\{E_1,\cdots,E_r\}$ is a subset of an orthonormal basis $\{E_1,\cdots,E_{\frac{n(n+1)}{2}}\}$ of $S^n$ and $0\neq\lambda_i\in R$.
Let $\hat{X}_1=\sum\limits_{i=1}x_iE_i$, where $x_i=\left<\hat{X}_1,E_i\right>$ for $i=1,2,\cdots,\frac{n(n+1)}{2}$. Since $\|\hat{X}_1\|_F$=1, we have $\sum\limits_{i=1}|x_i|^2=1$. Reorder the indices such that
\begin{equation}
\mid\lambda_1\mid\ge\mid\lambda_2\mid\ge\cdots\ge\mid\lambda_r\mid.
\end{equation}
Then we obtain
\begin{equation*}
\begin{aligned}
\mid\left<\mathcal{A},\hat{X}_1\circ\bar{\hat{X_1}}\right>\mid&=\mid\sum\limits_{i=1}^r\lambda_i\mid x_i\mid^2\mid\\
&\le \mid\lambda_1\mid.
\end{aligned}
\end{equation*}
On the other hand, the optimality leads to
\begin{equation*}
\begin{aligned}
\mid\left<\mathcal{A},\hat{X}_1\circ\bar{\hat{X_1}}\right>\mid&\ge\mid\left<\mathcal{A},E_1\circ\bar{E_1}\right>\mid\\
&= \mid\lambda_1\mid.
\end{aligned}
\end{equation*}
Hence,
\[\mid\lambda_1\mid\le\mid\left<\mathcal{A},\hat{X}_1\circ\bar{\hat{X_1}}\right>\mid=\mid\hat{\lambda}_1\mid\le\mid\lambda_1\mid.\]
So, $\mid\hat{\lambda}_1\mid=\mid\lambda_1\mid,~\mid x_1\mid=1$. Therefore, $\hat{X}_1=e^{i\theta}E_1,$ for any $\theta\in[0,2\pi]$,
 and
\[\hat{\lambda}_1=\left<\mathcal{A},\hat{X}_1\circ\bar{\hat{X_1}}\right>=\left<\mathcal{A},E_1\circ\bar{E}_1\right>=\lambda_1.\]
Then
let $x_1=1$, we have $\hat{X}_1=E_1$.
\end{proof}

Now, we prove Theorem \ref{thm:recover}.
\begin{proof}
By Lemma \ref{lem:onestep}, there exists $j\in{1,2,\cdots,r}$ such that $\hat{X}_1=E_j$, $\hat{\lambda}_1=\lambda_j$. Let
\[\mathcal{A}_1=\mathcal{A}-\hat{\lambda}_1\hat{X_1}\circ\bar{\hat{X_1}}=\sum\limits_{i\neq j}\lambda_iE_i\circ\bar{E}_i,\]
and
\[\hat{X}_2\in\underset{X\in S^n,\|X\|_F=1}{argmax}\mid\left<\mathcal{A}_1,X\circ\bar{X}\right>\mid,~\hat{\lambda}_2=\left<\mathcal{A},\hat{X}_2\circ\bar{\hat{X_2}}\right>.\]
By the similar proof of Lemma \ref{lem:onestep}, we know that there exists a $k\in\{1,2,\cdots,n\}\backslash\{j\}$ such that $\hat{\lambda}_2=\lambda_k,~\hat{X}_2=E_k$. Repeatedly, We can induce a permutation $\pi$ on $\{1,2,\cdots,r\}$ such that
\[\hat{\lambda}_j=\lambda_{\pi(j)},~\hat{X}_j=E_{\pi(j)},~j=1,2,\cdots,r.\]
\end{proof}
\begin{corollary}
Let \[\mathcal{A}=\sum\limits_{i=1}^r\lambda_iE_i\circ\bar{E_i},\]
where $\{E_1,\cdots,E_r\}$ is a subset of an orthonormal basis $\{E_1,\cdots,E_{\frac{n(n+1)}{2}}\}$ of $S^n$.  $0\neq\lambda_i\in R$ are different from each other, and are ordered as
\[\mid\lambda_1\mid\ge\mid\lambda_2\mid\ge\cdots\ge\mid\lambda_r\mid.\] Suppose $\{(\hat{\lambda}_i,\hat{X}_i)\}_{i=1}^n$ is the output of the SMROA algorithm for input $\mathcal{A}$. Then, $\hat{\lambda}_i=\lambda_i$. $\hat{X}_i=E_i$, for $i=1,2,\cdots,r$. In particular, if $\lambda_i>0$, we have $\hat{\lambda}_1>\hat{\lambda}_2>\cdots>\hat{\lambda}_r$; if $\lambda_i<0$, we have $\hat{\lambda}_1<\hat{\lambda}_2<\cdots<\hat{\lambda}_r$.
\end{corollary}
This proposition directly follows from the proof of Lemma \ref{lem:onestep}.
\begin{remark}
According to the proof in Lemma \ref{lem:onestep}, if $\mathcal{A}$ is a partial-symmetric tensor with entries whose imaginary part is not zero, the SMROA algorithm may fail. This is because that if $\hat{X}_1=e^{i\theta}X_1$, $\hat{X}_1\circ\hat{X}_1\neq X_1\circ X_1$.
\end{remark}
\section{Applications of Matrix Outer Product Model}\label{sec:application}
\subsection{Low-CP-Rank Tensor Completion}
The following theorem shows that the CP rank of the CPS tensor $\mathcal{A}$ can be bounded by $rank_M(\mathcal{A})$.
\begin{theorem}\label{thm:bound}
For CPS tensor $\mathcal{A}\in C^{n^4}=\sum\limits_{i=1}^{rank_M(\mathcal{A})}\lambda_i E_i\circ E_i$, it holds that
\[rank_M(\mathcal{A})\le rank_{CP}(\mathcal{A})\le r^2rank_M(\mathcal{A}),\]
where $r=\underset{i}{max}\{rank(E_i)\}$.
\end{theorem}
Scenarios like colored video data with static background is more likely to be a low-CP-rank tensor\cite{Jiang2018Low}. So the completion problem for this kind of data can be formulated as
\begin{equation}\label{equ:cp}
\begin{aligned}
min ~&rank_{CP}(\mathcal{X})\\
s.t.~&P_{\Omega}(\mathcal{X})=P_{\Omega}(\mathcal{A}).
\end{aligned}
\end{equation}
\eqref{equ:cp} is intractable to deal with directly, since the CP rank of a tensor is generally hard to estimate. Here we follow the idea of Jiang et al.\cite{Jiang2018Low} to cope with the completion problem of partial-symmetric tensor.
According to Theorem \ref{thm:bound}, we may approximate it by
\begin{equation}\label{equ:replace}
\begin{aligned}
min ~&rank_{M}(\mathcal{X})\\
s.t.~&P_{\Omega}(\mathcal{X})=P_{\Omega}(\mathcal{A}),
\end{aligned}
\end{equation}
The following example gives an intuitive explanation for the rationality of the approximation.
\begin{example}
Suppose $\mathcal{A}\in R^{10^4}$, $\mathcal{A}=\sum\limits_{i=1}^4\lambda_i(x_i\circ x_i\circ y_i\circ y_i+y_i\circ y_i\circ x_i\circ x_i)$ is partial-symmetric. Then $rank_M(\mathcal{A})\le rank_{CP}(\mathcal{A})\le 8$, while $M(\mathcal{A})\in R^{100\times 100}$ is a low-rank matrix.
\end{example}
\eqref{equ:replace} is relaxed as a low rank approximation problem with nuclear norm regular term,
\begin{equation}\label{equ:completion}
\begin{aligned}
min~ & \mu\|X\|_*+\frac{1}{2}\|P_{\Omega}(\mathcal{X})-P_{\Omega}(\mathcal{A})\|_F^2\\
s.t.~& X=M(\mathcal{X}).\\
\end{aligned}
\end{equation}
$\mathcal{X}$ in above problems is required to be partial-symmetric. And the sample set $\Omega$ is partial-symmetric.

We can apply the Fixed Point Continuation(FPC) algorithm\cite{Ma2011Fixed} to solve \eqref{equ:completion}.
\begin{algorithm}
    \caption{FPC for \eqref{equ:completion} }\label{alg:FPC}
     Given a PS tensor $\mathcal{A}\in R^{n^4}$. Initialize $\mathcal{X}_0=0$. Given parameter $\tau$, $\mu_1$, $\mu_L$
    \begin{algorithmic}
    \For {$\mu=\mu_1,\cdots,\mu_L$}
        \For {$k=1,2,\cdots$}
        \State $\mathcal{Y}_k=\mathcal{X}_{k-1}-\tau P_\Omega(\mathcal{X}_{k-1}-\mathcal{A}),$
        \State $M(\mathcal{X})_k=D_{\tau\mu}(M(\mathcal{Y})_{k}).$
        \EndFor
    \EndFor
    \State Return $\mathcal{X}_k$.
    \end{algorithmic}
\end{algorithm}

$\mu_{k+1}=\eta*\mu_k$, where $0<\eta<1$. The convergence of this algorithm is guaranteed \cite{Ma2011Fixed}. Since the iteration in Algorithm \ref{alg:FPC} does not change the symmetric property  of $X_k$ when $\Omega$ and $\mathcal{A}$ are partial-symmetric, the solution $X_*$ is still partial-symmetric.
\subsection{Low Rank Matrix Outer Product Approximation}
Parallel to the sparse rank-one approximation problem, we can also discuss the low rank matrix outer product "rank-one" approximation as follows based on the matrix decomposition model proposed in last section.
\begin{equation}\label{equ:sparse}
\begin{array}{lrc}
min~\|\mathcal{A}-\alpha X\circ X\|^2_F+\lambda\|X\|_*\\
s.t.\quad\alpha\in R,\\
\qquad X\in S^n,~\|X\|_F=1,
\end{array}
\end{equation}
where $\mathcal{A}\in R^{n^4}$ is partial-symmetric.

We modify the proximal linearized minimization algorithm (PLMA) proposed by Bolte et al.\cite{Bolte2014Proximal} to solve problem \eqref{equ:sparse}. The iterative scheme is
\begin{equation}\label{equ:iter}
\begin{array}{lrc}
\hat{X}_{k+1}\in\underset{X}{argmin}\{f(\alpha_k,X_k)+\left<X-X_k,\nabla_X f(\alpha_k,X_k)\right>+\frac{t_k}{2}\|X-X_k\|_F^2+\lambda\|X\|_*\},\\
X_{k+1}=\frac{\hat{X}_{k+1}}{\|\hat{X}_{k+1}\|_F},\\
\alpha_{k+1} = \left<\mathcal{A},X_{k+1}\circ X_{k+1}\right>.
\end{array}
\end{equation}
where $t_k>0$, $f(\alpha,X)=\|\mathcal{A}-\alpha X\circ X\|^2_F$ and
\begin{equation}
\nabla_X f(\alpha,X)=4\alpha^2\|X\|_F^2X-4\alpha\mathcal{A}X.
\end{equation}

To solve \eqref{equ:iter}, there exists simple singular value thresholding operator for the nuclear norm, see \cite{Cai2008A}.
\begin{lemma}\cite[Theorem 2.1]{Cai2008A}
Let $X\in R^{n_1\times n_2}$ be an arbitrary matrix and $U\Sigma V^T$ be its SVD. It is known that
$$\partial\|X\|_*=\{UV^T+W:~W\in R^{n_1\times n_2},~U^TW=0,~WV=0,~\|W\|_2\le1\}.$$
\begin{equation}
\begin{aligned}
D_{\tau}(X)&=\underset{Y}{argmin}~\frac{1}{2}\|Y-X\|_F^2+\tau\|Y\|_*\\
&=UD_{\tau}(\Sigma)V^T\\
&=U_0(\Sigma_0-\tau I)V_0^T,
\end{aligned}
\end{equation}
where $D_{\tau}(\Sigma)=diag((\sigma_i-\tau)_+)$, $U_0,~V_0$ are the singular vectors associated with singular values greater than $\tau$.
\end{lemma}
Thus we can compute the analytical solution of $\hat{X}_{k+1}$ in \eqref{equ:iter}, that is,
\begin{equation}\label{equ:iter_x}
\hat{X}_{k+1}=D_{\frac{\lambda}{t}}(X_k-\frac{1}{t}\nabla_X f(\alpha_k,X_k)).
\end{equation}

\begin{lemma}\label{lem:bounded}
The $\alpha$-sublevel set of the objective function of \eqref{equ:sparse}, $\{(\alpha,X)\in R\times S^n\mid f(\alpha,X)+\lambda\|X\|_*\le \alpha,~\|X\|_F=1\}$, is bounded.
\end{lemma}
It is obvious since $f(\alpha,X)+\lambda\|X\|_*\ge\mid\mathcal{A}-\mid\alpha\mid\|X\circ X\|_F\mid^2\rightarrow +\infty$ when $\mid\alpha\mid\rightarrow +\infty$.

For the iterative scheme, we have the following sufficient descent property.
\begin{lemma}\label{lem:descent}
Let $f$ be continuously diferentiable over $X$ and its gradient $\nabla_X f$ be $L_f$-Lipschitz continuous locally. Then for any $t_k>L_f$, it holds that
\[f(\alpha_{k+1},X_{k+1})+\lambda\|X_{k+1}\|_*\le f(\alpha_k,X_k)+\lambda\|X_k\|_*-\frac{t_k-L_f}{2}\|X_k-X_{k+1}\|^2.\]
\end{lemma}
\begin{proof}
Since $\nabla f$ is $L_f$-Lipschitz continuous locally, we have
\begin{equation}\label{equ:taylor}
f(\alpha_k,\hat{X}_{k+1})\le f(\alpha_k,X_k)+\left<\hat{X}_{k+1}-X_k,\nabla_X f(\alpha_k,X_k)\right>+\frac{L_f}{2}\|\hat{X}_{k+1}-X_k\|_F^2.
\end{equation}
According to \eqref{equ:iter}, we also obtain that
\begin{equation}\label{equ:equ1}
\left<\hat{X}_{k+1}-X_k,\nabla_X f(\alpha_k,X_k)\right>+\frac{t_k}{2}\|\hat{X}_{k+1}-X_k\|_F^2+\lambda\|\hat{X}_{k+1}\|_*\le \lambda\|X_k\|_*.
\end{equation}
Add \eqref{equ:taylor} and \eqref{equ:equ1}, we have
\[f(\alpha_{k},\hat{X}_{k+1})+\lambda\|\hat{X}_{k+1}\|_*\le f(\alpha_k,X_k)+\lambda\|X_k\|_*-\frac{t_k-L_f}{2}\|X_k-\hat{X}_{k+1}\|^2.\]
Since $\alpha_{k+1}=\left<\mathcal{A},{X}_{k+1}\circ{X}_{k+1}\right>$ minimizes
\[f(\alpha,X_{k+1})+\lambda\|X_{k+1}\|_*,\]
we obtain that
\[f(\alpha_{k+1},X_{k+1})+\lambda\|X_{k+1}\|_*\le f(\alpha_{k},\hat{X}_{k+1})+\lambda\|\hat{X}_{k+1}\|_*.\]
The desired inequlity then follows.
\end{proof}
\begin{theorem}
For the sequence $\{X_k\}$ generated by \eqref{equ:iter_x}, its any cluster $X_*$ is a stationary point of \eqref{equ:sparse}.
\end{theorem}
The proof is similar to that in \cite{0A,Wang2017Low}, we just omit it.

\section{The equivalence property of CPS tensors} \label{sec:rank1Equ}
In this section, we prove that a fourth-order CPS tensor, denoted as $\mathcal{T}
\in \mathbb{C}^{n^4}_{ps}$, is rank-one if and only if a specific matrix unfolding of $\mathcal{T}$ is rank-one. We first prove the following lemma.
\begin{lemma}
	\label{lemma:CPSdecomp}
	If $\mathcal{T} \in \mathbb{C}_{ps}^{n^4}$ if rank-one, then there exist $\lambda \in \mathbb{R}$ and $x \in
	\mathbb{C}^n$ such that $\mathcal{T}= \lambda x \otimes x \otimes \bar{x}
	\otimes \bar{x}$.
	
	\begin{proof}
		Since $\mathcal{T}$ is rank-one, we have
		\[ \mathcal{T}= x \circ y \circ w \circ z, \quad \mathrm{} x, y, w,
		z \in \mathbb{C}^n . \]
		According to the conjugate symmetry of $\mathcal{T}$, we obtain that
		$\mathcal{T}_{[1, 2 ; 3, 4]} = \mathrm{vec} (x \circ y) \circ
		\mathrm{vec} (w \circ z)$ is a Hermitian matrix, thus, there exists $\lambda \in
		\mathbb{R}$ such that $w \circ z = \lambda \bar{x} \circ \bar{y}$.
		This further implies that there are $\alpha, \beta \in \mathbb{C}$ such that $w =
		\alpha \bar{x}$, $z = \beta \bar{y}$ and $\alpha \beta =
		\lambda$. Therefore, we have
		\[ \mathcal{T}= x \circ y \circ w \circ z = \lambda x \circ y
		\circ \bar{x} \circ \bar{y} . \]
		On the other hand, $\mathcal{T}$ is symmetric about the first and the second indexes, so $x = y$ and
		\[ \mathcal{T}= \lambda x \circ x \circ \bar{x} \circ \bar{x} . \]
	\end{proof}
\end{lemma}

Then we prove that $\mathcal{T} \in \mathbb{C}^{n^4}_{ps}$
if rank-one if and only if $\mathcal{T}_{[3, 2 ; 1, 4]}$ if a rank-one matrix.

\begin{lemma}
	\label{lemma:rankEqv} Let $\mathcal{T} \in
	\mathbb{C}_{ps}^{n^4}$, if the unfolding matrix $\mathcal{T}_{[3, 2 ; 1,4]}$ is rank-one, then $\mathcal{T}$ is a rank-one tensor.
	
	\begin{proof}
		According to Lemma \ref{lemma:CPSdecomp}, $\mathcal{T}$ can be represented as
		\[ \mathcal{T}= \sum_{i = 1}^r \lambda_i a_i \circ a_i \circ \bar{a}_i
		\circ \bar{a}_i, \quad \lambda_i \in \mathbb{R}, a_i \in \mathbb{C}^n
		. \]
		Then
		\[ \mathcal{T}_{[3, 2, 1, 4]} = \sum_{i = 1}^r \lambda_i  \bar{a}_i
		\circ a_i \circ a_i \circ \bar{a}_i, \quad \mathcal{T}_{[3, 2 ;
			1, 4]} = \sum_{i = 1}^r \lambda_i \mathrm{vec} (\bar{a}_i \circ a_i)
		\circ \mathrm{vec} (a_i \circ \bar{a}_i) . \]
		The later is a rank-one Hermitian matrix, so there exists $Y \in \mathbb{C}^{n
			\times n}$ such that
		\[ \mathcal{T}_{[3, 2 ; 1, 4]} = \mathrm{vec} (Y) \circ \mathrm{vec}
		(\bar{Y}) \mathrm{} . \]
		Let $Y = \sum_{i = 1}^R \sigma_i x_i \circ y_i$ be the SVD of $Y$, then
		\[ \mathcal{T}_{[3, 2, 1, 4]} = \sum_{i = 1}^R \sum_{j = 1}^R \sigma_i
		\sigma_j x_i \circ y_i \circ \bar{x}_j \circ \bar{y}_j . \]
		Since the second and third indexes of $\mathcal{T}_{[3, 2, 1, 4]}$ are symmetric, we have
		\[ \mathcal{T}_{[3, 2, 1, 4]} =\mathcal{T}_{[3, 1, 2, 4]} = \sum_{i = 1}^R
		\sum_{j = 1}^R \sigma_i \sigma_j x_i \circ \bar{x}_j \circ y_i
		\circ \bar{y}_j . \]
		This means that
		\[ \mathcal{T}_{[3, 2 ; 1, 4]} = \sum_{i = 1}^R \sum_{j = 1}^R \sigma_i
		\sigma_j \mathrm{vec} (x_i \circ \bar{x}_j) \circ \mathrm{vec} (y_i
		\circ \bar{y}_j) . \]
		It is easy to verify that $\{ \mathrm{vec} (x_i \circ \bar{x}_j)\}_{i, j = 1}^R$ and
		$\{ \mathrm{vec} (z_i \circ \bar{z}_j)\}_{i, j =
			1}^R$ are orthogonal basis, hence $R = 1$. Otherwise, we will have $\tmop{rank}
		(\mathcal{T}_{[3, 2 ; 1, 4]}) = R^2 > 1$. Thus,
		\[ \mathcal{T}_{[3, 2, 1, 4]} = \sigma^2 x \circ \bar{x} \circ z
		\circ \bar{z}, \]
		that is, $\tmop{rank} (\mathcal{T}_{[3, 2, 1, 4]}) = \mathrm{rank}
		(\mathcal{T}) = \mathrm{rank}_{cps} (\mathcal{T}) = 1$.
	\end{proof}
\end{lemma}

\begin{remark}
	It is pointed that a fourth-order real tensor
	$\mathcal{X}$ is rank-one if and only if $\mathcal{X}_{[1, 2 ; 3, 4]}$ and
	$\mathcal{X}_{[3, 2 ; 1, 4]}$ are rank-one matrix {\cite[Lemma 3.1]{Yuning2016Rank}}. The above Lemma
	\ref{lemma:rankEqv} reinforces this result for the CPS tensor. However only
	$\mathcal{T}_{[1, 2 ; 3, 4]}$ being rank-one cannot guarantee a partial-symmetric tensor $\mathcal{T}$ being rank-one. For example, the following $\mathcal{T} \in
	\mathbb{R}^{2 \times 2 \times 2 \times 2}$ is a partial-symmetric tensor,
	\[ \mathcal{T}= e_1 \circ e_1 \circ e_1 \circ e_1 + e_1 \circ e_1 \circ e_2
	\circ e_2 + e_2 \circ e_2 \circ e_1 \circ e_1 + e_2 \circ e_2 \circ e_2
	\circ e_2 . \]
	According to $\mathcal{T}_{[1, 3 ; 2, 4]} = I \in \mathbb{R}^{4 \times 4}$,
	we have $\tmop{rank} (\mathcal{T}) > 1$. However,
	\[ \mathcal{T} [1, 2 ; 3, 4] = \left[ \begin{array}{cccc}
		1 & 0 & 0 & 1\\
		0 & 0 & 0 & 0\\
		0 & 0 & 0 & 0\\
		1 & 0 & 0 & 1
	\end{array} \right] \]
	is a rank-one matrix.
\end{remark}

\subsection{Applications to the best rank-one approximation problem}

This section concentrates on the best rank-one approximation problem,
\begin{equation}
	\label{model:rank1App} \min \|\mathcal{T}- \lambda x \circ x \circ
	\bar{x} \circ \bar{x} \|_F^2  \quad s.t. \lambda \in \mathbb{R}, x \in
	\mathbb{C}^n, \|x\| = 1,
\end{equation}
or its equivalent formulation
\begin{equation}\label{equ:temp}
	\max \mid\langle \mathcal{T}, x \circ x \circ
	\bar{x} \circ \bar{x} \rangle\mid  \quad s.t.x \in \mathbb{C}^n, \|x\| = 1,
\end{equation}
where $\mathcal{T} \in \mathbb{C}^{n^4}_{ps}$ is the conjugate partial-symmetric tensor. Solving \eqref{equ:temp} is further equal to solve two minimization problems,  $\min\left<\mathcal{T},x \circ x \circ
	\bar{x} \circ \bar{x}\right>$ and $\min\left<-\mathcal{T},x \circ x \circ
	\bar{x} \circ \bar{x}\right>$ over $\{x\in\mathbb{C}^n,\|x\|=1\}$.

Without loss of generality, we only consider the following problem,
\begin{equation}
	\label{model:rank1AppEqv} \min \langle -\mathcal{T}, x \circ x \circ
	\bar{x} \circ \bar{x} \rangle  \quad s.t.x \in \mathbb{C}^n, \|x\| = 1.
\end{equation}

\subsubsection{Convex relaxations of (\ref{model:rank1AppEqv})}

Let $\mathcal{X} \assign x \circ x \circ \bar{x} \circ \bar{x}$,
the (\ref{model:rank1AppEqv}) can be reformulated equivalently as
\begin{equation}
	\label{model:rank1AppEqv1} \min - \langle \mathcal{T}, \mathcal{X} \rangle
	\quad s.t.\mathcal{X} \in \mathbb{C}^{n^4}_{ps}, \mathrm{rank} (\mathcal{X})
	= 1, \|\mathcal{X}\|_F = 1.
\end{equation}

According to
\eqref{lemma:rankEqv}, the rank-one tensor constraint $rank(\mathcal{A})=1$ can be replaced by a rank-one matrix constraint, i.e.,
(\ref{model:rank1AppEqv1}) can be equivalently reformulated into

\begin{equation}
	\label{model:matRankCons} \min \langle -\mathcal{T}, \mathcal{X} \rangle
	\quad s.t.\mathcal{X} \in \mathbb{C}^{n^4}_{ps}, \mathrm{rank}
	(\mathcal{X}_{[3, 2 ; 1, 4]}) = 1, \|\mathcal{X}\|_F = 1.
\end{equation}

In recent years, it is popular to replace the rank function by the matrix nuclear norm. After relaxation, one expects to obtain a low rank solution via solving the nuclear norm based problems. The nuclear norm
$\| \cdummy \|_{\ast}$ is defined as the sum of singular values of a matrix
{\cite{2010Guaranteed}}. Following this line,  (\ref{model:matRankCons})
can be relaxed to the following nuclear norm regularized convex optimization problem
\begin{equation}
	\label{model:conv1} \min \langle -\mathcal{T}, \mathcal{X} \rangle + \rho
	\|\mathcal{X}_{[3, 2 ; 1, 4]} \|_{\ast}  \quad s.t.\mathcal{X} \in
	\mathbb{C}^{n^4}_{ps}, \|\mathcal{X}\|_F \leq 1,
\end{equation}
where $\rho > 0$ is a regularization parameter.
Analogously, we can also employ the nuclear norms as constraints, which results in the following problem
\begin{eqnarray}
	\min \langle -\mathcal{T}, \mathcal{X} \rangle  \quad s.t.\mathcal{X} \in
	\mathbb{C}^{n^4}_{ps}, \|\mathcal{X}_{[3, 2 ; 1, 4]} \|_{\ast} \leq 1.
	\label{model:conv2}
\end{eqnarray}

\paragraph{Determining whether a solution to (\ref{model:conv1}) or (\ref{model:conv2})
	is a global minimizer of (\ref{model:matRankCons}).}

{\cite{Yuning2016Rank}} proved that for $X \in \mathbb{R}^{M \times
	N}$, if $\|X\|_F = 1$ and $\|X\|_{\ast} = 1$, then
\tmtextrm{rank}$(X) = 1$. Based on this observation, we discuss in which case, the optimizer of
(\ref{model:conv1}) is the global minimizer of the original problem
(\ref{model:matRankCons}). Denote $\hat{p}$ as the optimal value of
(\ref{model:conv1}), the following corollary is presented.

\begin{corollary}
	\label{prop:solutionpro1}If $\hat{\mathcal{X}} \neq 0$
	is an optimal value to (\ref{model:conv1}), $\hat{p} \neq 0$, and
	$\mathcal{X}^{\ast} $ is a global minimizer of (\ref{model:matRankCons}), then
	\begin{enumerate}
		\item $\| \hat{\mathcal{X}} \|_F = 1$,
		
		\item $\| \hat{\mathcal{X}}_{[3, 2 ; 1, 4]} \|_{\ast} \geq 1$,
		
		\item if $\| \hat{\mathcal{X}}_{[3, 2 ; 1, 4]} \|_{\ast} = 1$, then
		\tmtextrm{rank}$_{\tmop{CP}} (\hat{\mathcal{X}}) = 1$,
		
		\item $\lambda^{\ast} \leq \hat{\lambda}$, where $\lambda^{\ast} =
		\langle \mathcal{T}, \mathcal{X}^{\ast} \rangle$, $\hat{\lambda} = \langle
		\mathcal{T}, \hat{\mathcal{X}} \rangle$.
	\end{enumerate}
\end{corollary}

\begin{proof}
	1. Since $\mathcal{X}= 0$ is a feasible solution to (\ref{model:conv1}), it holds that
	$\hat{p} < 0$. Suppose $\| \hat{\mathcal{X}} \|_F < 1$, then
	$\hat{\mathcal{X}} / \| \hat{\mathcal{X}} \|_F$
	is also a feasible solution and the objective value evaluated at it is $\hat{p}/\| \hat{\mathcal{X}} \|_F<\hat{p}$, which gives a contradiction to the optimality of $\hat{\mathcal{X}}$.
	
	\quad 2. Denote $\hat{\sigma}_i, i = 1, \ldots, r$ as the singular values of $\hat{\mathcal{X}}_{[3,
		2 ; 1, 4]}$, then according to $\| \hat{\mathcal{X}} \|_F = \sum_{i =
		1}^r \hat{\sigma}_i^2$ and $\| \hat{\mathcal{X}}_{[3, 2 ; 1, 4]} \|_{\ast} =
	\sum_{i = 1}^r \hat{\sigma}_i$, it can be seen that when $\| \hat{\mathcal{X}} \|_F =
	1$, $\| \hat{\mathcal{X}}_{[3, 2 ; 1, 4]} \|_{\ast} \geq \|
	\hat{\mathcal{X}} \|_F^2 = 1$.
	
	\quad 3. It follows from \eqref{lemma:rankEqv} immediately.
	
	\quad 4. Since $\mathcal{X}^{\ast}_{[3, 2 ; 1, 4]}$ is rank-one and
	$\|\mathcal{X}^{\ast} \|_F = 1$, we have $\|\mathcal{X}^{\ast}_{[3, 2 ; 1, 4]}
	\|_{\ast} = 1$ and $- \lambda^{\ast} + \rho \geq - \hat{\lambda} + \rho \|
	\hat{\mathcal{X}}_{[3, 2 ; 1, 4]} \|_{\ast}$. Combined with $\|
	\hat{\mathcal{X}}_{[3, 2 ; 1, 4]} \|_{\ast} \geq 1$, it is obtained that $\lambda^{\ast}
	\leq \hat{\lambda}$.
\end{proof}

Corollary \ref{prop:solutionpro1} implies that it is possible to identify whether an optimizer of
(\ref{model:conv1}) is an optimizer of (\ref{model:matRankCons})
by computing the sum of balanced nuclear norms. The following theorem summarizes this result.

\begin{theorem}
	Assume that $\mathcal{T} \neq 0$ and $\hat{\mathcal{X}}$ is a global minimizer of
	(\ref{model:conv1}). Then $\hat{\mathcal{X}}$ is a global optimizer of
	(\ref{model:matRankCons}) if and only if $\hat{\mathcal{X}}
	\neq 0, \| \hat{\mathcal{X}}_{[3, 2 ; 1, 4]} \|_{\ast} = 1$ and $\hat{p}
	\neq 0$.
\end{theorem}

Next we study (\ref{model:conv2}). The following observations also characterize the conditions for the optimizer of (\ref{model:conv2})
being the global optimizer of the original problem.

\begin{corollary}
	\label{prop:solutionpro2}Assume $\mathcal{T} \neq 0$. If
	$\hat{\mathcal{X}}$ is an optimal solution of (\ref{model:conv2}), then
	\begin{enumerate}
		\item  $\| \hat{\mathcal{X}}_{[3, 2 ; 1, 4]} \|_{\ast} = 1$,
		
		\item $\| \hat{\mathcal{X}} \|_F \leq 1$,
		
		\item if $\| \hat{\mathcal{X}} \|_F = 1$, then $\tmop{rank}_{\tmop{CP}}
		(\hat{\mathcal{X}}) = 1$.
	\end{enumerate}
\end{corollary}

\begin{proof}
	1. Since $\mathcal{T}/ \|\mathcal{T}_{[3, 2 ; 1, 4]} \|_{\ast}$
	is a feasible solution and the corresponding objective value is negative, so the optimal value of
	(\ref{model:conv2}) must be negative. Suppose $\|
	\hat{\mathcal{X}}_{[3, 2 ; 1, 4]} \|_{\ast} < 1$, then $\hat{\mathcal{X}}
	/ \| \hat{\mathcal{X}}_{[3, 2 ; 1, 4]} \|_{\ast}$
	is a feasible solution, whose objective value is smaller than the optimal value. It gives a contradiction.
	
	\quad 2. When $\| \hat{\mathcal{X}}_{[3, 2 ; 1, 4]} \|_{\ast} \leqslant
	1$, $\|\mathcal{X}\|_F^2 \leq \| \hat{\mathcal{X}}_{[3, 2 ; 1, 4]}
	\|_{\ast} \leqslant 1$.
	
	\quad 3. It follows from  Lemma  \ref{lemma:rankEqv}.
\end{proof}

Based on Corollary \ref{prop:solutionpro2}, we have the following theorem.

\begin{theorem}
	Assume $\mathcal{T} \neq 0$, then the optimizer $\hat{\mathcal{X}}$ of (\ref{model:conv2})
	is also a global optimizer of (\ref{model:matRankCons})
	if and only if $\| \hat{\mathcal{X}} \|_F = 1$.
\end{theorem}

\paragraph{Implementation of \eqref{model:conv1} via ADMM.}
By introducing a variable $\mathcal{y}$, problem \eqref{model:conv1} can be formulated as
\begin{equation}\label{equ:admm_conv1}
\min\langle-\mathcal{T},\mathcal{X}\rangle+\rho\|\mathcal{Y}_{[3,2;1,4]}\|_*\quad
s.t. \mathcal{X}\in\mathbb{C}_{ps}^{n^4},\|\mathcal{X}\|_F\le 1,~\mathcal{Y}=\mathcal{X}.
\end{equation}
Denote the Lagrangian function of \eqref{equ:admm_conv1} as
$$\mathcal{L}(\mathcal{X},\mathcal{Y},\Lambda)=
\langle-\mathcal{T},\mathcal{X}\rangle+\rho\|\mathcal{Y}_{[3,2;1,4]}\|_*+
\langle\Lambda,\mathcal{Y-X}\rangle+\frac{\tau}{2}\|\mathcal{Y}-\mathcal{X}\|_F^2.$$
The ADMM iteration steps for \eqref{equ:admm_conv1} are given as follows
\begin{align}
\mathcal{X}^{k+1}&=\underset{\|\mathcal{X}\|_F\le 1,\mathcal{X}\in\mathbb{C}_{ps}^{n^4}}{\arg\min}\mathcal{L}(\mathcal{X},\mathcal{Y}^k,\Lambda^k)
=\frac{P_{cps}[\frac{1}{\tau}(\tau\mathcal{Y}^k+\mathcal{T}+\Lambda^k)]}{\|P_{cps}[\frac{1}{\tau}(\tau\mathcal{Y}^k+\mathcal{T}+\Lambda^k)]\|_F}\\
\mathcal{Y}^{k+1}&=\arg\min\mathcal{L}(\mathcal{X}^{k+1},\mathcal{Y},\Lambda^k);\\
\Lambda^{k+1}&=\Lambda^k+\tau(\mathcal{Y}^{k+1}-\mathcal{X}^{k+1}).
\end{align}
To compute the subproblem of $\mathcal{Y}$, we first compute
\[\mathcal{Y}_{[3,2;1,4]}^{k+1}=D_{\frac{\rho}{\tau}}[(\mathcal{X}-\frac{1}{\tau}\Lambda^k)_{[3,2;1,4]}],\]
then fold the matrix $\mathcal{Y}_{[3,2;1,4]}^{k+1}$ back into $\mathcal{Y}^{k+1}$. $P_{cps}(\cdot)$ denotes the projection onto the set $\mathbb{C}_{ps}^{n^4}$, which is given by the following lemma.

\begin{lemma}
	\label{lemma:proj}Let $\mathcal{Y} \in \mathbb{C}^{n \times n \times n
		\times n}$,
	\begin{equation*}
		 P_{cps} (\mathcal{Y}) =
		 \frac{1}{8}  (\mathcal{Y}+\mathcal{Y}_{[1, 2, 4, 3]} +\mathcal{Y}_{[2,
			1, 3, 4]} +\mathcal{Y}_{[2, 1, 4, 3]} + \overline{\mathcal{Y}_{[3, 4, 1,
				2]}} + \overline{\mathcal{Y}_{[4, 3, 1, 2]}} + \overline{\mathcal{Y}_{[3,
				4, 2, 1]}} + \overline{\mathcal{Y}_{[4, 3, 2, 1]}}) .
	\end{equation*}
\end{lemma}
\begin{proof}
	It is easy to verify that for any $\mathcal{Z} \in \mathbb{C}^{n^4}_{\tmop{ps}}$,
	$\langle \mathcal{Y}, \mathcal{Z} \rangle = \langle
	\mathcal{Z}, \mathcal{Y} \rangle = \langle P_{\tmop{cps}} (\mathcal{Y}),
	\mathcal{Z} \rangle$. Thus for any $\mathcal{Z} \in
	\mathbb{C}^{n^4}_{\tmop{ps}}$, we have
	\begin{eqnarray*}
		\|\mathcal{Y}-\mathcal{Z}\|_F^2 & = & \|\mathcal{Y}- P_{\tmop{cps}}
		(\mathcal{Y}) + P_{\tmop{cps}} (\mathcal{Y}) -\mathcal{Z}\|^2_F\\
		& = & \|\mathcal{Y}- P_{\tmop{cps}} (\mathcal{Y})\|_F^2 +
		\|P_{\tmop{cps}} (\mathcal{Y}) -\mathcal{Z}\|_F^2\\
		& \geq & \|\mathcal{Y}- P_{\tmop{cps}} (\mathcal{Y})\|_F^2 .
	\end{eqnarray*}
	That is, $P_{cps} (\mathcal{Y}) = \arg \min_{\mathcal{Z} \in
		\mathbb{C}^{n^4}_{ps}} \|\mathcal{Z}-\mathcal{Y}\|_F^2$.
\end{proof}
\subsubsection{A nonconvex relaxation of (\ref{model:rank1AppEqv})}

In this section, we consider a nonconvex relaxation of (\ref{model:rank1AppEqv}). By introducing an auxiliary variable $\mathcal{y}$, (\ref{model:matRankCons}) can be reformulated as
\begin{equation}
	\label{model:NonCovOri} \min \frac{1}{2}  \|\mathcal{T}-\mathcal{X}\|_F^2
	\quad s.t.\mathcal{X}=\mathcal{Y}, \mathcal{X} \in \mathbb{C}^{n^4}_{ps},\tmop{rank} (\mathcal{X}_{[1, 2 ; 3, 4]}) = 1,
	\tmop{rank} (\mathcal{Y}_{[3, 2 ; 1, 4]}) = 1.
\end{equation}

By relaxing the constraint $\mathcal{X}=\mathcal{Y}$ and imposing it on the objective function, we obtain the following nonconvex relaxation of (\ref{model:NonCovOri})
\begin{equation}
	\label{model:NonCovRela} \min F (\mathcal{X}, \mathcal{Y}) \assign
	\frac{1}{2}  \|\mathcal{T}-\mathcal{Y}\|_F^2 + \frac{\rho}{2}
	\|\mathcal{X}-\mathcal{Y}\|_F^2  \quad s.t.\mathcal{X} \in
	\mathbb{C}^{n^4}_{ps}, \tmop{rank} (\mathcal{X}_{[1, 2 ; 3, 4]}) = 1, \tmop{rank} (\mathcal{Y}_{[3, 2 ; 1, 4]}) = 1,
\end{equation}
where $\rho > 0$ is a regularization parameter. It is possible to employ the alternating minimization scheme to solve (\ref{model:NonCovRela}), namely, optimize
$\mathcal{X}$ and $\mathcal{Y}$ alternatively as

\begin{align}
	\mathcal{Y}^{k + 1} & =\underset{{rank(\mathcal{Y}_{[3,2;1,4]})=1}}{\arg \min} F
	(\mathcal{X}^{k}, \mathcal{Y})\\
	\mathcal{X}^{k + 1} & = \underset{\mathcal{X}\in\mathbb{C}_{ps}^{n^4},\tmop{rank} (\mathcal{X}_{[1,2;3,4]}) = 1}{\arg \min} F
	(\mathcal{X}, \mathcal{Y}^{k+1}).
\end{align}

In the above, the subproblem about $\mathcal{Y}$ is a simple matrix rank-one approximation problem, which can be solved by SVD efficiently. It can be simplified as
\begin{equation}
	\mathcal{Y}^{k + 1} = \underset{{rank(\mathcal{Y}_{[3,2;1,4]})=1}}{\arg \min}
 \|\frac{2}{1 + \rho}  \left( \frac{1}{2} \mathcal{T}+ \frac{\rho}{2}
	\mathcal{X}^{k} \right) -\mathcal{Y}\|_F^2,
\end{equation}
Assume that $\mathcal{X}^k$ is a CPS tensor, then $\mathcal{Y}^{k+1}_{[3,2,1,4]}$ can be expressed as $\mathcal{Y}^{k+1}_{[3,2,1,4]}=\lambda E\circ \bar{E}$, where $E$ is a Hermitian matrix according to \cite[Theorem 1]{de2007fourth}. We further denote $E=\sum_{i=1}\sigma_ia_i\circ \bar{a}_i$ as the SVD of $E$. Then $$\mathcal{Y}=\sum_{i=1}\sum_{j=1}\lambda\sigma_i\sigma_j\bar{a}_i\circ\bar{a}_j\circ a_i\circ a_j.$$
Thus $\mathcal{X}^{k+1}=\lambda\sigma_1\sigma_1\bar{a}_1\circ\bar{a}_1\circ a\circ a$ is always a rank-one CPS tensor.

\section{Numerical Results}\label{sec:numerical}
In this section, we present some numerical experiments to corroborate the results we obtain in previous sections. We conduct all the following algorithms in MATLAB (Realease 2016b) and performed them on a Lenovo laptop with an Intel(R) Core(TM) Processor with access to 8GB of RAM.
\subsection{Recover CPS Tensor with SMROA Algorithm}
For the real partial-symmetric tensor, we first construct it as $\mathcal{A}=\sum\limits_{i=1}^3\lambda_i E_i\circ E_i$, where $E_i$ are $n\times n~(n=3,4,\cdots,7)$ mutually orthogonal and $\lambda_i$ are generated randomly and are different from each other. We find that the output satisfies that $\mid\hat{\lambda}_1\mid\ge\mid\hat{\lambda}_2\mid\ge\mid\hat{\lambda}_3\mid$ and $X_i=\pm E_{\pi(i)}$. An example is given as follows.
\begin{example}
For $\mathcal{A}=0.0547E_1\circ E_1-0.1048E_2\circ E_2-0.2209E_3\circ E_3$, $\hat{\lambda}_1=-0.2209$, $\hat{X}_1=-E_3$, $\hat{\lambda}_2=-0.1048$, $\hat{X}_1=-E_2$, $\hat{\lambda}_3=0.0547$, $\hat{X}_1=-E_3$.
\end{example}
Then we randomly generated 100 partial-symmetric tensors, the numerical results show that SMROA algorithm recover the matrix outer product decomposition successfully. The output $\{\hat{\lambda}_i,\hat{X}_i\}_{i=1}^r$ satisfy that $\mid\hat{\lambda}_1\mid\ge\mid\hat{\lambda}_2\mid\ge\cdots\ge\\\mid\hat{\lambda}_r\mid$. $\hat{X}_i$ are mutually orthogonal.

For the conjugate partial-symmetric tensor, we also design similar experiments as the real case. And the results are still consistent with our theories. The following is an example.
\begin{example}
$\mathcal{A}=20.6777E_1\circ \bar{E}_1+16.1910E_2\circ \bar{E}_2+7.6104E_3\circ \bar{E}_3-6.7274E_4\circ \bar{E}_4-4.7920E_5\circ \bar{E}_5+2.7811E_6\circ \bar{E}_6$. $E_i$ are complex symmetric matrices and mutually orthogonal. Then the SMROA algorithm generates $\hat{\lambda}_i$ exactly the same as $\lambda_i$ in order.
\end{example}
\subsection{Low Rank Tensor Completion with Nuclear Norm}
Since we are not quite sure what kinds of real data will be partial-symmetric tensor, we use the synthetic data as examples for low-CP-rank partial-symmetric tensor completion problem.

We generate randomly the PS tensor as
\begin{equation}\label{equ:generate}
\mathcal{A}=\sum\limits_{i=1}^r\lambda_i (x_i\circ x_i\circ y_i\circ y_i+y_i\circ y_i\circ x_i\circ x_i),
\end{equation}
then $rank_{CP}(\mathcal{A})\le 2r$. We then recover it by \eqref{equ:completion} and report the average of the relative errors, which is defined as
\[Err:=\frac{\|\mathcal{X}^*-\mathcal{A}\|_F}{\|\mathcal{A}\|_F},\]
where $\mathcal{X}^*$ is the optimizer of \eqref{equ:completion}. The stop criteria for the inner iteration is $\frac{\|X_{k+1}-X_k\|_F}{\|X_k\|_F}< 10^{-10}$. Part of our numerical results are presented in Table \ref{tab:tab1}. The err denotes the average relative errors of 20 instances for each $(n,r,p)$ pair, where $n$ is the dimension of $\mathcal{A}$, $r$ is the same as \eqref{equ:generate} and $p$ is the sample ratio. The number in the $rank_M$ column with star superscript is the maximum $rank_M(\mathcal{\mathcal{X}_*})$ of the solution in the numerical examples.
\begin{table}[!htbp]
\caption{Results of low rank tensor completion}\label{tab:tab1}
\centering
\begin{tabular}{|l|cc|cc|cc|}
\hline
p   & err          & $rank_M$  & err          & $rank_M$  & err          & $rank_M$  \\ \hline
\multicolumn{7}{|c|}{n=10}                                                           \\ \hline
    & \multicolumn{2}{c|}{r=1} & \multicolumn{2}{c|}{r=2} & \multicolumn{2}{c|}{r=3} \\ \hline
0.8 & 1.1305e-8    & 2         & 7.4016e-9    & 4         & 2.8517e-9    & 6         \\ \hline
0.5 & 3.4894e-8    & 2         & 1.0411e-08   & $6^*$     & 1.0883e-08   & 6         \\ \hline
\multicolumn{7}{|c|}{n=15}                                                           \\ \hline
    & \multicolumn{2}{c|}{r=1} & \multicolumn{2}{c|}{r=2} & \multicolumn{2}{c|}{r=3} \\ \hline
0.8 & 1.0548e-08   & 2         & 3.0970e-09   & 4         & 1.6526e-09   & 6         \\ \hline
0.5 & 5.5183e-09   & 2         & 3.2419e-09   & 4         & 2.9300e-09   & 6         \\ \hline
\end{tabular}
\end{table}
From Table \ref{tab:tab1}, we find that most of our examples satisfy that $rank_{CP}(\mathcal{A})=rank_M(\mathcal{A})=2r$ except the cases for $(n=10, r=2, p=0.5)$.

\subsection{Rank-one approximation for CPS tensor}
The CPS tensor $\mathcal{T}$ in this experiment are generated as $\mathcal{T}=\sum_{i=1}^5a_i\circ a_i\circ \bar{a}_i\circ\bar{a}_i$, where every $a_i$ is a randomly generated complex vector. The dimensions are 5. We generated 50 instances. We implement the ADMM for the convex relaxation \eqref{model:conv1} and the ALM for the nonconvex relaxation \eqref{model:NonCovRela}. We first apply ALM for \eqref{model:NonCovRela} within 5 iterations, denote the output as $\mathcal{X}$, and obtain $\bar{\mathcal{X}}=\mathcal{X}/\|\mathcal{X}\|_F$. Let $\rho=\left<\mathcal{T},\bar{\mathcal{X}}\right>$ in \eqref{model:conv1}. Then, we apply ADMM for \eqref{model:conv1}. We observe that it finds a rank-one CPS tensor successfully, that is, the global optimum of \eqref{model:rank1AppEqv}, at average 21.7400 steps.

\section{Conclusions}
In this paper, the matrix outer product decomposition for fourth-order CPS tensors was studied. It was shown that the SROA algorithm can be used to recovery the matrix outer product decomposition exactly. Then, we applied this decomposition for the low-CP-rank tensor completion. We also studied the rank-one equivalence property of fourth-order CPS tensors, which build the relationship between rank-one CPS tensor and a certain unfolded matrix of the CPS tensor. Based on this property, different relaxation approaches were proposed to solving the best rank-one approximation problem for the CPS tensor. Numerical experiments demonstrated the usefulness of the matrix outer product decomposition model.
\backmatter

%
%

\section*{Declarations}
\begin{itemize}
\item Funding. This work was funded by the National Natural Science Foundation of China (Grant No. 11671217, No. 12071234).
\item Conflict of interest/Competing interests. The authors declare that they have no conflict of interest.
\end{itemize}

\bibliography{ref}


\end{document}